\documentclass[11pt,a4paper]{amsart}
\usepackage{amssymb}
\usepackage{amsmath}
\usepackage{amsfonts,a4wide}
\usepackage{bbm}
\usepackage{enumerate}
\usepackage[latin1]{inputenc}
\usepackage{shadow}
\usepackage{slashed}
\usepackage[all]{xy}
\usepackage{tikz-cd}
\usepackage{soul}
\usepackage[letterpaper, margin=2.7cm]{geometry}
\usepackage{cleveref}
\usepackage{bbm}
\usepackage{mathtools}

\newcommand\CC{\mathbb{C}}

\DeclareMathOperator {\image}{image}

\theoremstyle{definition}
\newtheorem{theorem}[subsection]{Theorem}

\newtheorem{proposition}[subsection]{Proposition}

\newtheorem{corollary}[subsection]{Corollary}
\newtheorem{lemma}[subsection]{Lemma}

\newtheorem{remark}[subsection]{Remark}            % put [subsection] between the two for normal mode
\newtheorem{example}[subsection]{Example}

\newtheorem*{conjecture*}{Conjecture}

\title{On the behavior of Massey products under field extension}
\author{Aleksandar Milivojevi\'c}
\address{University of Waterloo, Faculty of Mathematics, Waterloo, Ontario} 
\email{amilivoj@uwaterloo.ca}
\begin{document}

\begin{abstract} We show that global vanishing of Massey products on a commutative differential graded algebra is not invariant under field extension. Non-vanishing triple Massey products remain non-vanishing upon field extension, while higher Massey products can generally vanish. If the field being extended is algebraically closed, all non-vanishing Massey products remain non-vanishing on a finite type commutative differential graded algebra.
\end{abstract}

\maketitle

\section{Introduction}

Massey products are higher-order multi-valued operations on the cohomology of a differential graded algebra which provide obstructions to formality, i.e. the existence of a chain of quasi-isomorphisms between the differential graded algebra and its cohomology equipped with trivial differential. We will be working with $\mathbbm{k}$-cdga's, i.e. commutative (though this property will largely be inessential) differential graded algebras over fields $\mathbbm{k}$. For purposes of discussion we will assume the fields are of characteristic zero and the cdga's connected; however, our general results (\Cref{main} (2) and (4)) do not require these assumptions.

Following the sign convention of \cite{Kraines}, for homogeneous cohomology classes $z_1, \ldots, z_n$, we first define a \emph{defining system} for the $n$-fold Massey product $\langle z_1, \ldots, z_n \rangle$ to be a choice of representatives $a_{i, i}$ for $z_i$, and for each pair $1\leq i < j\leq n$ other than $(i,j) = (1,n)$, a choice of element $a_{i,j}$ (if it exists) such that $$d(a_{i,j}) = \sum_{k=i}^{j-1} (-1)^{|a_{i,k}|} a_{i,k}a_{k+1,j}.$$ Then the $n$-fold Massey product is the set of cohomology classes $$\Bigl\{ \,\, \Bigl[ \sum_{k=1}^{n-1} (-1)^{|a_{1,k}|} a_{1,k}a_{k+1,n} \Bigr] \,\, \Bigr\}$$ obtained by running over all defining systems\footnote{The triple Massey product $\langle z_1, z_2, z_3 \rangle$ enjoys the special property that it can equivalently be described, when defined, as a single element in the quotient of the cohomology modulo the ideal generated by $z_1, z_3$.}. The Massey product is well-defined if the above set is non-empty, i.e. if there exists at least one defining system. We say the Massey product is \emph{trivial}, or \emph{vanishes}, if $0 \in \langle z_1, \ldots, z_n \rangle$. The above set is a quasi-isomorphism invariant of differential graded algebras; we point the reader to \cite[Proposition 2.18]{FCMF} for a careful proof.

Though formality of a cdga implies that all Massey products vanish, the converse does not hold (see for example \cite[1.5]{CN} for a study of this phenomenon on sufficiently highly connected rational Poincar\'e duality algebras). Here we illustrate one more shortcoming of the property ``all Massey products vanish'': unlike what is clearly true for formality, the validity of this property is not preserved under field extension. 

We then investigate how individual Massey products behave under field extension. Concretely, non-trivial triple Massey products, like cup products, remain non-trivial upon field extension. For higher Massey products this need not be the case, unless the starting field is algebraically closed:

\begin{theorem}\label{main} We have the following: \begin{enumerate} \item There are examples of $\mathbbm{k}$-cdga's on which all Massey products vanish, but such that upon field extension not all Massey products vanish (\Cref{global}). \item Non-trivial triple Massey products remain non-trivial upon field extension (\Cref{triple}). \item A non-trivial quadruple (or higher) Massey product can in general become trivial upon field extension (\Cref{vanishoverC}). \item If $\mathbbm{k}$ is algebraically closed, non-trivial Massey products of any order remain non-trivial upon extension of the field $\mathbbm{k}$ on a degree-wise finite-dimensional cdga. (\Cref{algclosed}). \end{enumerate}  \end{theorem}

In order to detect the non-formality of a cdga via Massey products, one may thus benefit from looking at both larger and smaller ground fields. The example used in (1), and its minimal models, provide examples of non-formal $\mathbb{R}$-cdga's with ``uniformly'' (i.e. simultaneously, consistently, in a precise sense) vanishing real Massey products (\Cref{uniformvanishing}).

In the category of $A_{\infty}$- (or $C_{\infty}$-) algebras, there are quasi-isomorphisms $H(A) \to A$, where the cohomology $H(A)$ is an $A_{\infty}$-algebra with trivial differential, the multiplication it inherits from $A$, and $n$-to-1 operations $\{m_n\}_{n\geq 3}$. The higher operations on $H(A)$ are also sometimes referred to as Massey products, and enjoy the property of being genuine $n$-to-1 operations. However, the $A_{\infty}$-algebra structure on $H(A)$ is not unique; automorphisms of this structure will in general change the operations $m_n$, and after collecting the outputs of all possible $m_n$ on a given input, one again ends up with a multi-valued operation as in the ``ad hoc'' definition given before. Note also that the ad hoc Massey products are not even well-defined on all inputs, unlike the operations $m_n$. We refer the reader to \cite{BMFM} for an investigation of the relation between these two notions of Massey products. In what follows we will be considering only the ad hoc notion, defined directly on the cdga level. To fix terminology, we say that an element $a$ in a dga such that $da = b$ is a primitive for $b$.

\subsection*{Acknowledgements} This work was inspired by a conversation with Scott Wilson about \cite[Section 6]{StelzigWilson}; I thank him and Jonas Stelzig for numerous helpful discussions and comments, together with the Max Planck Institute for Mathematics in Bonn for its generous hospitality. I am furthermore grateful to the referee for their careful and helpful comments.

\section{Global vanishing of Massey products is not preserved under field extension}\label{global}

In this section we show by example that having all Massey products vanish over a given field does not imply all Massey products vanish over a larger field. Upon field extension, elements in an algebra generally may become decomposable, allowing for substantially new Massey products to be considered\footnote{Cf. \cite[p. 203f.]{Positselski} on how the \emph{tensor} Massey products, arising from the Eilenberg--Moore spectral sequence for an augmented dga, generally have a larger domain of definition than the ad hoc (\emph{tuple}, in the terminology of loc. cit.) Massey products.}. Our example will have all Massey products vanishing over the real numbers, while its non-formality will be detected by a non-trivial triple Massey product on its complexification.

The example is based on one often considered at the interface of rational homotopy theory and complex geometry. Take the complex Lie group $G$ consisting of matrices of the form $$\begin{pmatrix} 1 & x & z \\ 0 & 1 & y \\ 0 & 0 & 1 \end{pmatrix}, $$ with $x, y, z \in \mathbb{C}$. Quotienting by the subgroup of matrices with entries in the Gaussian integers $\mathbb{Z}[i]$ yields a compact complex threefold known as the Iwasawa manifold. The holomorphic one-forms $\phi_1 = dx, \phi_2 = dy, \phi_3 = xdy - dz$ on $\CC^3$ are left $G$-invariant and hence descend to the Iwasawa manifold. The induced map from the exterior algebra generated by these forms and their conjugates $\Lambda(\phi_1, \overline{\phi_1}, \phi_2, \overline{\phi_2}, \phi_3, \overline{\phi_3})$, equipped with the de Rham differential determined by $d\phi_3 = \partial \phi_3 = \phi_1 \phi_2$, into the de Rham algebra of smooth complex-valued forms on the Iwasawa manifold, is a quasi-isomorphism by a theorem of Nomizu. From this finite-dimensional model of the Iwasawa manifold one easily sees that it carries a non-trivial triple Massey product over the complex numbers. Namely consider $\langle [\phi_1], [\phi_1], [\phi_2] \rangle$. For any choice of primitive of $\phi_1 \phi_1 = 0$ and of $\phi_1 \phi_2$, the resulting class in the Massey product is $[\phi_1 \phi_3] \neq 0$.

Now let us consider the above over the real numbers. As a real Lie group, the group $G$ consists of matrices of the form $$\begin{pmatrix} 1 & 0 & x_1 & -y_1 & x_3 & -y_3 \\ 0 & 1 & y_1 & x_1 & y_3 & x_3 \\ 0 & 0 & 1 & 0 & x_2 & -y_2 \\ 0 & 0 & 0 & 1 & y_2 & x_2 \\ 0 & 0 & 0 & 0 & 1 & 0 \\ 0 & 0 & 0 & 0 & 0 & 1 \end{pmatrix},$$ where $x_i, y_i \in \mathbb{R}$.

By looking at the entries of $Q^{-1}dQ$ for a generic matrix $Q$ of this form, we compute a real basis of left-invariant one-forms to be given by $$\eta_1 = dx_1, \eta_2 = dy_1, \eta_3 = dx_2, x_4 = dy_2, \eta_5 = dx_3 + y_1dy_2 - x_1dx_2, \eta_6 = x_1dy_2 + y_1dx_2 - dy_3,$$ and so by Nomizu's theorem the natural inclusion of the $\mathbb{R}$-cdga $$A = \left( \Lambda(\eta_1, \eta_2, \eta_3, \eta_4, \eta_5, \eta_6), d\eta_1 = d\eta_2 = d\eta_3 = d\eta_4 = 0, d\eta_5 = \eta_1\eta_3 - \eta_2\eta_4, d\eta_6 = \eta_2\eta_3 + \eta_1\eta_4) \right)$$ into the smooth real-valued forms on the Iwasawa manifold is a quasi-isomorphism, cf. \cite[Example 6.24]{StelzigWilson} and \cite[Section 6]{CFG}. The complexification of this cdga is identified with the complex model given above via $\phi_j = \eta_{2j-1} + i\eta_{2j}, \overline{\phi_j} = \eta_{2j-1} - i\eta_{2j}$ for $j=1,2,3$. 

\begin{lemma}\label{products} Let $A$ be the real model of the Iwasawa manifold given above. If $z_1 z_2 = 0$ for non-zero classes $z_1, z_2 \in H^1(A)$, then $z_2 = c z_1$ for some real number $c$. \end{lemma}

\begin{proof} Note that $H^1(A)$ is spanned by $[\eta_1], [\eta_2], [\eta_3], [\eta_4]$. Choosing representatives $$\alpha_1 \eta_1 + \alpha_2 \eta_2 + \alpha_3 \eta_3 + \alpha_4 \eta_4, \,\,\, \beta_1 \eta_1 + \beta_2 \eta_2 + \beta_3 \eta_3 + \beta_4 \eta_4$$ of $z_1, z_2$, we have that $z_1 z_2$ is represented by \begin{align*} (\alpha_1 \beta_2 - \alpha_2 \beta_1) \eta_1\eta_2 &+ (\alpha_1\beta_3 - \alpha_3\beta_1) \eta_1\eta_3 + (\alpha_1\beta_4 - \alpha_4 \beta_1)\eta_1\eta_4 + (\alpha_2 \beta_3 - \alpha_3 \beta_2)\eta_2\eta_3 \\ &+ (\alpha_2\beta_4 - \alpha_4\beta_2)\eta_2\eta_4 + (\alpha_3 \beta_4 - \alpha_4 \beta_3)\eta_3\eta_4. \end{align*}

The image of $d$ in degree two is spanned by $\eta_1\eta_3 - \eta_2\eta_4$ and $\eta_2\eta_3 + \eta_1\eta_4$, so $z_1 z_2 = 0$ is equivalent to \begin{align} \alpha_1 \beta_3 - \alpha_3 \beta_1 + \alpha_2 \beta_4 - \alpha_4 \beta_2 &= 0, \\ \alpha_1 \beta_4 - \alpha_4 \beta_1 - \alpha_2 \beta_3 + \alpha_3 \beta_2 &= 0, \\ \alpha_1 \beta_2 - \alpha_2 \beta_1 &= 0, \\ \alpha_3 \beta_4 - \alpha_4 \beta_3 &= 0. \end{align}

\textit{Case 1: $\alpha_1 = 0$}. Then (3) gives $\alpha_2 \beta_1 = 0$. 

\textit{Case 1.1:} If $\alpha_2 = 0$, then (1) and (2) become \begin{align*} \alpha_3 \beta_1 + \alpha_4 \beta_2 &= 0, \\ -\alpha_4 \beta_1 + \alpha_3 \beta_2 &= 0, \end{align*} i.e. the scalar product of $(\beta_1, \beta_2)$ with both $(\alpha_3, \alpha_4)$ and $(-\alpha_4, \alpha_3)$ is zero. Since the latter two are orthogonal, we conclude $(\alpha_3, \alpha_4) = 0$ or $(\beta_1, \beta_2) = 0$. In the first case we would have $z_1 = 0$ so we are done. In the second case, we have that both $z_1$ and $z_2$ are represented by elements in the span of $\eta_3, \eta_4$, and the claim clearly holds. 

\textit{Case 1.2:} If $\beta_1 = 0$, then $z_1$ and $z_2$ are represented by elements in the span of $\eta_2, \eta_3, \eta_4$, and the claim again clearly holds (since there is no $\eta_1$ involved, which necessarily shows up in any non-trivial differential). 

\textit{Case 2: $\alpha_1 \neq 0$.} We can assume $\alpha_1 = 1$. So, by (3), $\beta_2 = \alpha_2 \beta_1$. Now (1) gives us $$\beta_3 = (\alpha_3 + \alpha_2 \alpha_4)\beta_1 - \alpha_2 \beta_4.$$ Plugging this into (2), we get $(1+ \alpha_2^2)\beta_4 = \alpha_4(1+\alpha_2^2)\beta_1$, hence $\beta_4 = \alpha_4 \beta_1$, as $1 + \alpha_2^2 \neq 0$ since we are working over the real numbers. Lastly, (4) gives us $\alpha_4 \beta_3 = \alpha_3 \alpha_4 \beta_1$. If $\alpha_4 \neq 0$, we conclude $\beta_3 = \alpha_3 \beta_1$; if $\alpha_4 = 0$, then from (1) and (2) we see $\beta_3 = \alpha_3 \beta_1$. Hence $z_2 = \beta_1z_1$. \end{proof}

\begin{corollary}\label{iwasawamassey} Every real $n$-fold Massey product $\langle z_1, \ldots, z_n \rangle$ on the Iwasawa manifold $M$, for $n\geq 3$ with $z_i \in H^1(M;\mathbb{R}) \cong H^1(A)$, is trivial. \end{corollary}

\begin{proof} Since $z_i z_{i+1} = 0$ by assumption, by the above lemma $z_{i+1}$ is a scalar multiple of $z_i$ (we may assume all the $z_i$ to be non-zero, since otherwise the Massey product automatically vanishes). Note that the differential is trivial on degree zero, so each class in $H^1(A)$ in fact has a unique representative. Therefore the representatives of $z_i$ are scalar multiples of each other. In particular, the pairwise products of representatives of $z_i$ and $z_{i+1}$ are zero, and we can choose the zero element as primitive. Inductively choosing zero for all primitive elements, we are done. \end{proof}

\begin{remark} There are non-trivial real Massey products on $A$ landing in $H^{\geq 3}(A)$, cf. \cite[Section 6]{CFG}, for example $\langle [\eta_1], [\eta_3\eta_4], [\eta_2] \rangle$. \end{remark}

%\vspace{2em}

We can truncate the Iwasawa manifold's real minimal model $A$ in order to obtain an $\mathbb{R}$-cdga with cohomology concentrated in degrees up to two, and with vanishing Massey products. Namely, consider the differential ideal $A^{\geq 3}$ of elements of degrees $\geq 3$, and consider the quotient $$B = A/A^{\geq 3}.$$ It is immediate that the quotient map $A \xrightarrow{f} B$ is a 1-quasi-isomorphism, i.e. an isomorphism on $H^1$ and an injection on $H^2$. 

\begin{corollary}\label{masseyB} All real Massey products on the $\mathbb{R}$-cdga $B$ vanish. \end{corollary}

\begin{proof} Since the quotient map $A \xrightarrow{f} B$ is a 1-quasi-isomorphism, all real Massey products involving only degree 1 classes can be computed on $A$ \cite[Section 3.6]{Suciu}, where they vanish by \Cref{iwasawamassey}. For degree reasons, all other Massey products (involving at least one element of degree at least two) trivially vanish, since $H^{\geq 3}(B) = 0$.  \end{proof}

\begin{proposition}\label{Bnotformal} $B$ is not formal.\end{proposition}

Before giving the argument, let us recall some concepts and results. Let $\mathbbm{k}$ be a field of characteristic zero. A 1-minimal model (over $\mathbbm{k}$) \cite[Definition 5.3]{Morgan} of a $\mathbbm{k}$-cdga $A$ is a minimal cdga generated in degree 1, with a 1-quasi-isomorphism to $A$. Recall, a 1-quasi-isomorphism is a morphism of cdga's inducing an isomorphism on $H^1$ and an injection on $H^2$. The 1-minimal model is unique up to isomorphism \cite[Theorem 5.6]{Morgan}. We say a cdga is 1-formal if there is zig-zag of 1-quasi-isomorphisms connecting it to its cohomology equipped with trivial differential. Equivalently, a cdga is 1-formal if there is a 1-quasi-isomorphism from its 1-minimal model to its $\mathbbm{k}$-cohomology algebra \cite[Lemma 2.2]{Macinic}. We refer the reader to \cite{Macinic}, where this notion is also referred to as \emph{1-stage formality}, for further discussion.

Formality famously satisfies a descent property: a degree-wise cohomologically finite-dimensional connected $\mathbbm{k}$-cdga $A$ is formal in the category of $\mathbbm{k}$-cdga's if and only if $A \otimes_{\mathbbm{k}} \mathbb{F}$ is formal in the category of $\mathbb{F}$-cdga's, where $\mathbbm{k}$ is any field extension of $\mathbbm{k}$. The analogous statement holds for 1-formality, and more generally, $i$-formality (or \emph{$i$-stage formality}, in the terminology of \cite{Macinic}). We refer the reader to \cite[Theorem 12.1]{Infinitesimal} and \cite[Theorem 4.19]{SuciuWang}, and to \cite[Section 3]{Suciu} for a nice overview.

Back to the truncated model of the Iwasawa manifold, there is clearly a $\mathbb{Q}$-cdga $B'$ such that $B' \otimes_{\mathbb{Q}} \mathbb{R} \cong B$, as the same holds for the model $A$ of the Iwasawa manifold. So, $B$ is formal as an $\mathbb{R}$-cdga if and only if $B'$ is formal as a $\mathbb{Q}$-cdga. We remark that since $B$ is obtained by extending a $\mathbb{Q}$-cdga, by taking nilpotent models we can realize the phenomena in this section by topological spaces. 

\vspace{0.5em} 
\noindent \emph{Proof of \Cref{Bnotformal}.} If $B$ were formal, then it would be 1-formal. Note that $A$ is a real 1-minimal model of $B$ via the quotient map; the 1-minimal model of a cdga is unique up to isomorphism. Therefore we would have a 1-quasi-isomorphism $A \xrightarrow{f} H(A)$ \cite[Lemma 2.2]{Macinic}. Tensoring with $\mathbb{C}$, we would have a 1-quasi-isomorphism from the complex minimal model of the Iwasawa manifold to its $\mathbb{C}$-valued cohomology. Hence all Massey products landing in $H^2(A\otimes \mathbb{C})$ would be trivial \cite[Proposition 3.15]{Suciu}, a contradiction (recall $\langle [\phi_1], [\phi_1], [\phi_2] \rangle$). 

For the reader's convenience, we spell the last sentence out for triple Massey products. Consider a triple Massey product $\langle [a], [b], [c] \rangle$ where $[a],[b],[c] \in H^1(A \otimes \mathbb{C})$. Choose primitives $x$ and $y$ of $ab$ and $bc$ respectively. Then there are closed elements $x', y' \in A$ such that $[x'] = f(x)$ and $[y'] = f(y)$. Since $f$ is a 1-quasi-isomorphism, there are \emph{closed} elements $\tilde{x}, \tilde{y} \in A$ such that $f(\tilde{x}) = [x'], f(\tilde{y}) = [y']$. Then $$ d(x - \tilde{x}) = ab, \,\,\, d(y - \tilde{x}) = bc, \,\,\, f(x-\tilde{x}) = 0, \,\,\, f(y-\tilde{y}) = 0.$$ With this new choice of primitives, namely $x - \tilde{x}$ and $y - \tilde{y}$ instead of $x$ and $y$, the Massey product $\langle [a], [b], [c] \rangle$ is represented by $(-1)^{|b|+1}[(x-\tilde{x})c + a(y-\tilde{y})]$. Now, $$f^*[(x-\tilde{x})c + a(y-\tilde{y})] = [f(x-\tilde{x})f(c) + f(a)f(y-\tilde{y})] = 0.$$ Since $f$ induces an injective map on $H^2$, we conclude that the Massey product is trivial. An analagous argument applies to quadruple and higher Massey products. \qed

\begin{remark}\label{uniformvanishing} This $\mathbb{R}$-cdga $B$ satisfies more than just vanishing of all real Massey products: There exists a choice of representing forms for all classes such that in any well-defined (real) Massey product, one can simultaneously choose the zero element for all primitives. This illustrates that the explanation of the criterion for formality given in \cite[Theorem 4.1]{DGMS}\footnote{One should assume in this criterion that the minimal cdga is furthermore in \emph{normal form} as in \cite{Papadima}.} as ``a way of saying that \emph{one may make uniform choices so that the forms representing all Massey products and higher order Massey products are exact}'' is not meant to go both ways.

Indeed, let us elaborate on this. Picking any section for the projection $\ker(d_B) \to H(B)$, i.e. a cocycle-choosing homomorphism, that sends $[\eta_i]$ to $\eta_i$, we can take the zero element whenever a choice of primitive must be made when constructing Massey products, and all Massey products will be exact with this ``uniform'' (or, ``simultaneous'') choice. 

Of course, $B$ is not minimal, so the criterion in \cite[Theorem 4.1]{DGMS} does not directly apply. Regardless, since we can build a minimal model $M(B)$ for $B$ by adding only generators of degree $\geq 2$, we can still take the zero element whenever a choice of primitive must be made in constructing a Massey product landing in degree two (i.e. a Massey product of degree one classes). Since $H^{\geq 3}(M(B)) = 0$, any choice of section $d_{M(B)}^{-1}:\operatorname{image}(d) \to M(B)$ of the differential $d_{M(B)}: M(B) \to \operatorname{image}(d)$ will suffice as a primitive-choosing homomorphism yielding trivial Massey products. Namely, first denote by $s$ any cocycle-choosing homomorphism such that $s[\eta_i] = \eta_i$. Now, for any Massey product $\langle [a_{1,1}],...,[a_{n,n}]\rangle$ in $M(B)$, one inductively builds a defining system yielding the zero class by setting $$a_{i,j}:=d^{-1}\sum_{i\leq l<j}(-1)^{|a_{i,l}|}a_{i,l}a_{l+1,j},$$ where $a_{i,i} = s[a_{i,i}] \in M(B)$.

\end{remark}

As a corollary of our calculations we obtain the following example, involving well-studied manifolds, demonstrating that even though formality satisfies field descent, the existence of a quasi-isomorphism between two cdga's does not have to:

\begin{example} (cf. \cite[Remark II.2.15]{Raghunathan}) The Iwasawa manifold and the product $H \times H$ of the Heisenberg manifold with itself have the same complex homotopy type, but distinct real homotopy types. Said differently, their de Rham algebras of real-valued forms are not connected by a chain of $\mathbb{R}$-cdga quasi-isomorphisms, while their de Rham algebras of complex-valued forms are connected by a chain of $\mathbb{C}$-cdga quasi-isomorphisms.

Indeed, a complex minimal model of $H \times H$ is given by complexifying the real minimal model given by $$\left( \Lambda(x_1, x_2, x_3, y_1, y_2, y_3), dx_3 = x_1x_2, dy_3 = y_1y_2 \right),$$ and so relabelling $x_i$ to $\phi_i$ and $y_i$ to $\overline{\phi_i}$ identifies this with the complex minimal model of the Iwasawa manifold given earlier. The real homotopy types are distinct, since the real minimal model of $H$ has a non-trivial Massey product in $H^2$, e.g. $\langle [x_1], [x_1], [x_2] \rangle$. \end{example}

\section{Behavior of individual Massey products under field extension}

We now investigate the behavior of Massey products under field extension, establishing parts (2)--(4) of \Cref{main}.

\subsection{Triple Massey products persist under field extension}\label{triple} For simplicity of notation let us consider the field extension $\mathbb{R} \subset \mathbb{C}$. Consider a triple Massey product $\langle [x], [y], [z] \rangle$ in an $\mathbb{R}$-cdga $A$. If it is trivial, then it is clearly trivial in $A \otimes \mathbb{C}$ as well. We will now establish that if the triple Massey product is trivial on $A \otimes \mathbb{C}$, then it was trivial on $A$ to begin with. That is, a non-trivial triple Massey product on $A$ remains non-trivial on $A \otimes \mathbb{C}$. 

So, suppose that $[x], [y], [z] \in H(A)$ are such that the triple Massey product $\langle [x], [y], [z] \rangle$ is trivial in $A\otimes \mathbb{C}$ That is, there are $\alpha, \beta, \Psi \in A \otimes \mathbb{C}$ such that $$d\alpha = xy, \,\, d\beta = yz, \,\, d\Psi = \alpha z - (-1)^{|x|}x\beta.$$ Choose a homogeneous real vector space basis $\{u_j\}$ for $A$; then a real basis for $A \otimes \mathbb{C}$ is given by $\{u_j, i u_j\}$. Write $$\alpha = \sum_j c_j u_j + i\sum_j c_j' u_j, \,\,\,\, \beta = \sum_j \tilde{c}_j u_j + i \sum_j \tilde{c}'_j u_j,$$ where $c_j, c'_j, \tilde{c}_j, \tilde{c}'_j \in \mathbb{R}$. Since $d$ is the complexification of the differential on $A$, we conclude $$d(\sum_j c_j u_j) = xy \,\,\,\,\,\,\, \mathrm{ and } \,\,\,\,\,\,\, d(\sum_j \tilde{c}_j u_j) = yz.$$ Since $\alpha z - (-1)^{|x|}x\beta = d\Psi$, we similarly conclude that $$(\sum_j c_j u_j) z - (-1)^{|x|} x(\sum_j \tilde{c}_j u_j)$$ is exact by some real element. That is, the real Massey product $\langle [x], [y], [z] \rangle$ is trivial in $A$. 

\vspace{0.5em}

The above argument works just as well for any field extension $\mathbbm{k} \subset \mathbb{F}$, by choosing a $\mathbbm{k}$-basis $\{1, c_i\}$ of $\mathbb{F}$ and taking the $\mathbbm{k}$-basis $\{u_j, c_iu_j\}_{i,j}$ for $A \otimes_{\mathbbm{k}} \mathbb{F}$. We give an alternative argument in \Cref{PDproof}. However, the result does not generalize to quadruple (or higher) Massey products.

\subsection{Non-trivial quadruple Massey products can become trivial upon field extension}\label{vanishoverC} 

In this section, we give examples of $\mathbbm{k}$--cdga's $A$ such that a specific quadruple Massey product is trivial on $A$, yet it is trivial on $A \otimes_\mathbbm{k} \mathbb{F}$ for some field extension $\mathbbm{k} \subset \mathbb{F}$.

To begin, consider the $\mathbb{R}$-cdga $$\left( \Lambda(x, y, a, b, u, v, w), dx = dy = db = 0, da = xy, du = ay, dv = by, dw = 2xu - a^2 - b^2 \right),$$ where $|x| = 2, |y| = 3, |a| = |b| = 4, |u| = |v| = 6, |w| = 7$. Consider the Massey product $\langle [x], [y], [y], [x] \rangle$. For the unique representatives $x$ and $y$, a generic choice of primitives is given by $xy = d(a+k_1b), y^2 = d(k_2xy), yx = d(a+k_3b)$ for some scalars $k_i$. Then the triple Massey product representatives of $\langle [x], [y], [y] \rangle$ and $\langle [y], [y], [x] \rangle$ are made exact via \begin{align*} (a+k_1b) y - k_2 x^2 y &= d(u + k_1v - k_2xa + k_4 x^3 + k_5xb), \\k_2x^2y + ya + k_3by &= d(k_2xa + u + k_3v + k_6x^3 + k_7xb).\end{align*}
The resulting element in the quadruple Massey product is then represented by \begin{align*} &(u+k_1v - k_2xa + k_4x^3 + k_5xb)x - (a+k_1b)(a+k_3b) + x(k_2xa + u + k_3v + k_6x^3 + k_7xb) \\ &= (2xu - a^2 - k_1k_3 b^2) + (k_1+k_3)(xv - ab) + (k_4+k_6)x^4 + (k_5 + k_7)x^2b. \end{align*}

We compute that $H^8$ is spanned by $\{ [x^4], [2xu - a^2], [x^2 b], [xv - ab] \}$. For the Massey product to be trivial, we need to choose $k_i$ so that $k_4+k_6 = 0, k_5 + k_7 = 0, k_1+k_3 = 0, k_1k_3 = 1$. This can be solved over $\mathbb{C}$ by choosing $k_1 = i, k_3 = -i$, but cannot be solved over $\mathbb{R}$. 

Completely analogously we have the following: let $\mathbbm{k} \subset \mathbbm{k}(\sqrt{\theta})$ be a proper extension of fields of characteristic zero. Then the $\mathbbm{k}$-cdga $$\left( \Lambda(x_2, y_3, a_4, b_4, u_6, v_6, w_7), dx = dy = db = 0, da = xy, du = ay, dv = by, dw = 2xu - a^2 + \theta b^2 \right)$$ has a non-trivial quadruple Massey product, namely $\langle [x], [y], [y], [x] \rangle$, which vanishes upon field extension to $\mathbbm{k}(\sqrt{\theta})$.

Note that there is a non-trivial triple Massey product in the above examples, on the unextended cdga, given by $\langle [x], [y], [b] \rangle$.

\subsection{Extending from an algebraically closed field} As can be seen in the example above, triviality of a Massey product comes down to solvability of a system of polynomial equations in the coefficients along some vector space basis of the considered cdga. This leads us to the following:

\begin{proposition}\label{algclosed} Let $\mathbb{K}$ be an algebraically closed field, and $\mathbb{K} \subset \mathbb{F}$ any extension. Let $A$ be a $\mathbb{K}$-cdga which is degree-wise finite-dimensional. If a Massey product $\langle [x_1], \ldots, [x_n] \rangle$ is non-trivial on $A$, then it remains non-trivial on $A \otimes_{\mathbb{K}} \mathbb{F}$. \end{proposition}

\begin{proof}
In each graded piece $A^i$ of $A$ choose a complement $C^i$ to $\image(d)$ inside $\ker(d)$, and choose a complement $I^i$ to $\ker(d)$ in $A^i$. This gives us a splitting $A^i = \image(d)^i \oplus C^i \oplus I^i$ for each $i$. Choose graded $\mathbb{K}$-vector space bases $\{u_i\}, \{v_i\}, \{w_i\}$ of $\image(d), C = \oplus_i C^i, I = \oplus_i I^i$. For convenience let us denote the union of these bases by $\{b_i\}$. Here and throughout, indices on lower-case letters are for enumerative purposes and do not correspond to the degree. There are scalars $\beta_{i,j}^k \in \mathbb{K}$ such that $b_i b_j = \sum_k \beta_{i,j}^k b_k$. 

Now, $d$ is an isomorphism $I \to \image(d)$; we denote its inverse by $\delta$. Extending the field to $\mathbb{F}$ respects the above splitting; $\{u_i \}, \{v_i\}, \{w_i \}$ still form bases and (the extended) $d$ is an isomorphism $I\otimes_{\mathbb{K}} \mathbb{F} \to \image(d) \otimes_{\mathbb{K}} \mathbb{F}$ with inverse the extension of $\delta$. 

We go through the procedure of building a generic representative of $\langle [x_1], \ldots, [x_n] \rangle$. First of all, a generic primitive for $x_i x_{i+1}$ is given by $$\delta(x_i x_{i+1}) + \sum_j \alpha_{i,i+1}^{u_j} u_j + \sum_j \alpha_{i,i+1}^{v_j} v_j.$$ Then a generic representative of $\langle [x_i], [x_{i+1}], [x_{i+2}] \rangle$ is given by \begin{align*} (-1)^{|x_{i+1}| + 1}\Bigl(\delta(x_i x_{i+1}) &+ \sum_j \alpha_{i,i+1}^{u_j} u_j + \sum_j \alpha_{i,i+1}^{v_j} v_j\Bigr) x_{i+2} \\ &- (-1)^{|x_i|} x_i\Bigl( \delta(x_{i+1} x_{i+2}) + \sum_{j'} \alpha_{i+1,i+2}^{u_{j'}} u_{j'} + \sum_{j'} \alpha_{i+1,i+2}^{v_{j'}} v_{j'}\Bigr).\end{align*} Expanding $x_i, x_{i+1}, x_{i+2}$ in terms of the basis, we see that exactness of this expression is equivalent to the vanishing of the coefficients along $\{v_i\}$ and $\{w_i\}$, which are $\mathbb{K}$-linear expressions in the $\alpha$'s.  Given that this element is exact, a generic primitive is given by the following: we apply $\delta$ (which writes each $u_i$ in terms of $w_j$) and add an element in $\ker(d)$, i.e. a linear combination of $\{u_i, v_j\}$, whose coefficients we also label with $\alpha$ and treat as variables. When considering the generic representative of the fourfold product $\langle [x_i], [x_{i+1}], [x_{i+2}], [x_{i+3}]\rangle$, exactness will be equivalent to the existence of a zero of a system of polynomial equations, in the variables $\alpha$ with coefficients in $\mathbb{K}$, and of degree $\leq 2$ ($\alpha$ terms will be multiplied with other $\alpha$ terms when multiplying the primitive of $x_i x_{i+1}$ with that of $x_{i+2}x_{i+3}$). 

Repeating the above, we see that triviality of the Massey product over $\mathbb{F}$ is equivalent to the existence of a zero $\{a_j\}$ of a system of $\mathbb{K}$-polynomial equations $\{P_i(\{\alpha_j\})\}$ over $\mathbb{F}$ (recall, the coefficients $\beta_{i,j}^k$ and the coefficients in the expansion of each $z_i$ are in $\mathbb{K}$). By the degree-wise finite-dimensionality assumption, the set of variables $\{\alpha_j\}$ and the set of polynomials $\{P_i\}$ under consideration are finite. If there were a zero over $\mathbb{F}$, then there would be one over its algebraic closure $\overline{\mathbb{F}}$. Now by the weak Nullstellensatz, this is equivalent to the ideal generated by the $P_i$ in $\mathbb{K}[\{\alpha_j\}]$ being proper. Since $\mathbb{K}$ is itself algebraically closed, this is in turn equivalent to the existence of a zero over $\mathbb{K}$.  
\end{proof}

In fact, the above is showing a bit more: if a Massey product of $\mathbb{K}$-classes is trivial over $\mathbb{F}$, then it is first of all \emph{well-defined}, and furthermore trivial over $\mathbb{K}$. 

For Massey products of length $2n$ or $2n+1$, the degrees of the polynomials that appear above are bounded from above by $n$. In this regard the persistence of non-triviality of triple Massey products upon field extension has the same explanation as the persistence of the non-triviality of a cup product upon field extension; namely, a $\mathbbm{k}$-linear system has a zero over an extension $\mathbb{F}$ only if it has a zero over $\mathbbm{k}$.  

Due to quasi-isomorphism invariance of Massey products, the degree-wise finite-dimensionality assumption above can be relaxed to cohomological degree-wise finite-dimensionality in common situations, e.g. if we are in characteristic zero and the cdga is modelling a simply connected space with degree-wise finite-dimensional cohomology.

Let us now turn back to the persistence of triple Massey products under field extension. One can extend any cohomologically finite-dimensional and connected $\mathbbm{k}$--cdga $A$ to one satisfying $n$-dimensional Poincar\'e duality, which we call its Poincar\'e dualization $P_n(A)$ \cite[Section 3]{MSZ} (see also the earlier \cite{Lambrechts}). Furthermore, a degree one map of $n$--dimensional Poincar\'e duality cdga's preserves triple Massey products \cite{Taylor}. Using this we give a topologically inspired argument for the claim of \Cref{triple}. We can take $A$ to be degree-wise finite-dimensional, and truncate it above a sufficiently high degree, not altering the triviality of a given Massey product.

\begin{proposition}\label{PDproof} Non-trivial triple Massey products on a cohomologically finite-dimensional and connected $\mathbbm{k}$-cdga remain non-trivial under field extension $\mathbbm{k} \subset \mathbb{F}$. \end{proposition}

\begin{proof} Let $\langle z_1, z_2, z_3 \rangle$ be a non-trivial triple Massey product on the $\mathbbm{k}$-cdga $A$. We will use the notation of \cite[Section 3]{MSZ}: We consider the (shifted) dual complex $D_nA$ given by $(D_nA)^k:=(A^{n-k})^\vee$ with differential $(D_nA)^k \to (D_nA)^{k+1}$ given on pure-degree elements $\varphi\in D_n A$ by $d(\varphi)(a):=(-1)^{|\varphi|-1}\varphi(da)$ for $a\in A$. The $n$-th \emph{Poincar\'e dualization} of $A$ is $P_nA:= A\oplus D_nA$, with multiplication extending that on $A$ and defined on pure-degree elements $a\in A$, $\varphi\in D_nA$ by the dual complex element given by $(\varphi\wedge a)(b)=\varphi(a\wedge b)$. For $\varphi,\psi\in D_n A$ we set $\varphi\wedge\psi=0$.

Now, consider the $\mathbbm{k}$-linear map $A \xrightarrow{i} A \otimes_{\mathbbm{k}} \mathbb{F}$ sending $a \mapsto a \otimes 1$. For any $n$ we obtain a $\mathbbm{k}$-linear map $\Phi$ from the shifted dual complex $D_n(A)$ to the shifted dual complex $D_n(A\otimes_{\mathbbm{k}} \mathbb{F})$, where the latter consists of the $\mathbb{F}$-linear functionals, thought of as a $\mathbbm{k}$-complex. Namely, we send $\varphi$ to the $\mathbb{F}$-linear functional $\Phi(\varphi)$ determined by $\Phi(\varphi)(a \otimes u) = \varphi(a)u$. Now the $\mathbbm{k}$-linear map $A \oplus D_n(A) \xrightarrow{i \oplus \Phi} (A \otimes_{\mathbbm{k}} \mathbb{F}) \oplus (D_n(A \otimes \mathbb{F}))$ is a map of $\mathbbm{k}$-cdga's. The verification is similar to \cite[Lemma 3.8]{MSZ}; we carry it out here. For clarity we denote the cup product on the Poincar\'e dualization by $\wedge$. The only non-trivial check of multiplicativity is for elements of the form $\varphi \wedge a$, where $\varphi$ is in the dual complex and $a \in A$. On the one hand, for $b \in A$ and $u \in \mathbb{F}$, we have \begin{align*} \left( (i \oplus \Phi) (\varphi \wedge a) \right)(b \otimes u) &= \Phi(\varphi \wedge a)(b \otimes u) \\ &= \left((\varphi \wedge a)(b)\right)u = \left( \varphi(ab) \right) u \\ &= \Phi(\varphi)(ab \otimes u) = \Phi(\varphi)(i(a) (b\otimes u)).\end{align*} On the other hand we have \begin{align*} \left( (i \oplus \Phi)(\varphi) \wedge (i \oplus \Phi)(a) \right) (b\otimes u) &= \left( \Phi(\varphi) \wedge i(a) \right) (b\otimes u) = \Phi(\varphi)(i(a)(b \otimes u)). \end{align*}

Therefore, for large enough $n$ we have a map of $\mathbbm{k}$-cdga's which satisfy cohomological Poincar\'e duality $P_n^{\mathbbm{k}}(A) \to P_n^{\mathbb{F}}(A \otimes_{\mathbbm{k}} \mathbb{F})$, where the superscript indicates the category in which the Poincar\'e dualization is performed. Note also that since $i$ maps 1 to 1, the volume class of $P_n^{\mathbbm{k}}(A)$ is mapped to the volume class of $P_n^{\mathbb{F}}(A \otimes_{\mathbbm{k}} \mathbb{F})$.

Now we argue that non-trivial triple Massey products remain non-trivial under such a map, following an argument due to Taylor \cite{Taylor}. A triple product involving cohomology classes on a cdga algebra is non-trivial if and only if it is non-trivial on the Poincar\'e dualization \cite[Proposition 3.10]{MSZ}\footnote{The article \cite{MSZ} mostly focuses on the case of characteristic, but the Poincar\'e dualization construction and this cited result hold in arbitrary characteristic}. For simplicity let us denote by $f$ the map induced by $i \oplus \Phi$ on $\mathbbm{k}$-cohomology followed by the inclusion of $\mathbbm{k}$-cohomology into $\mathbb{F}$-cohomology. 

Consider the ideal $\mathcal{J}_{z_1, z_3}$ generated by $z_1, z_3$ in $H(P_n^{\mathbbm{k}}(A))$, and the vector space $\mathcal{A}_{z_1, z_3}$ of classes $z_0$ such that $z_1 z_0 = z_0 z_3 = 0$, \cite[Notation 1.2]{Taylor}. We then have a map $$\langle z_1, - , z_3 \rangle : \mathcal{A}_{z_1, z_3} \to H(P_n^{\mathbbm{k}}(A))/\mathcal{J}_{z_1, z_3}.$$ For $z_0 \in \mathcal{A}_{z_1, z_3}$, the product $z_0 \langle z_1, z_2, z_3 \rangle$ is a single class \cite[Theorem 2.1]{Taylor}, since $z_0$ kills the indeterminacy. As a consequence, $z_0\langle z_1, z_2, z_3 \rangle$ being a non-zero class implies that $\langle z_1, z_2, z_3 \rangle$ is non-trivial. Now, the Poincar\'e duality pairing on cohomology induces a non-degenerate pairing \cite[Proposition 5.1]{Taylor} $$(H(P_n^{\mathbbm{k}}(A))/\mathcal{J}_{z_1, z_3})^r \otimes_{\mathbbm{k}} (\mathcal{A}_{z_1, z_3})^{n-r} \to \mathbbm{k}.$$ Therefore, given the non-trivial Massey product $\langle z_1, z_2, z_3 \rangle$, there is a $t \in \mathcal{A}_{z_1, z_3}$ such that $$\int t \langle z_1, z_2, z_3 \rangle = 1$$ under the Poincar\'e duality pairing \cite[Theorem 5.2]{Taylor}. Now consider the triple Massey product $\langle f(z_1), f(z_2), f(z_3) \rangle$ in $A \otimes_{\mathbbm{k}} \mathbb{F}$. We have that $f(t)$ is in the $\mathbb{F}$-vector space $\mathcal{A}_{f(z_1), f(z_3)}$, and so $f(t) \langle f(z_1), f(z_2), f(z_3) \rangle$ is a single class which is non-zero, since $f$ maps the volume class of $P_n^{\mathbbm{k}}(A)$ to that of $P_n^{\mathbb{F}}(A \otimes_{\mathbbm{k}} \mathbb{F})$. Therefore the Massey product $\langle f(z_1), f(z_2), f(z_3) \rangle$ is non-trivial on $A\otimes_{\mathbbm{k}} \mathbb{F}$. \end{proof}

One can compare the general non-persistence of quadruple and higher products under field extension with the non-preservation of such products under non-zero degree maps of rational Poincar\'e duality algebras as investigated in \cite{MSZ}.

\end{document}